\documentclass[12pt,reqno]{amsart}
\usepackage[left=80pt,right=80pt]{geometry}
\usepackage[usenames]{color}
\usepackage{amsmath}
\usepackage{amssymb}
\usepackage{amsthm}
\usepackage{float}
\usepackage{tikz}

\usepackage{hyperref,url}
\hypersetup{
	colorlinks=true,
  linkcolor=black,          % color of internal links (change box color with linkbordercolor)
  citecolor=black,         % color of links to bibliography
  filecolor=black,      % color of file links
  urlcolor=black           % color of external links
}

\newtheorem{prop}{Proposition}
\newtheorem{lemma}[prop]{Lemma}
\newtheorem{theorem}[prop]{Theorem}

\theoremstyle{definition}

\newtheorem{example}[prop]{Example}

\allowdisplaybreaks
\newcommand{\mylabel}[2]{#2\def\@currentlabel{#2}\label{#1}}
\begin{document}
\tikzset{mystyle/.style={matrix of nodes,
        nodes in empty cells,
        row 1/.style={nodes={draw=none}},
        row sep=-\pgflinewidth,
        column sep=-\pgflinewidth,
        nodes={draw,minimum width=1cm,minimum height=1cmanchor=center}}}
\tikzset{mystyleb/.style={matrix of nodes,
        nodes in empty cells,
        row sep=-\pgflinewidth,
        column sep=-\pgflinewidth,
        nodes={draw,minimum width=1cm,minimum height=1cmanchor=center}}}

\title{Further results on staircase graph words}

\author[SELA FRIED]{Sela Fried$^*$}
\thanks{$^*$Department of Computer Science, Israel Academic College,
52275 Ramat Gan, Israel.
\\
\href{mailto:friedsela@gmail.com}{\tt friedsela@gmail.com}}

\begin{abstract}
 \noindent
Staircase words are words in which consecutive letters do not differ by more than $1$. We generalize this by extending the restriction to letters lying further apart from each other and obtain the corresponding generating functions, which we express in terms of the Chebyshev polynomials of the second kind.
\bigskip
 
\smallskip

\noindent
\textbf{Math.~Subj.~Class.:} 68R05, 05A05, 05A15.
\end{abstract}

\maketitle

\baselineskip=0.20in

\section{Introduction}
Let $[k]^n$ be the set of all words of length $n$ over the alphabet $[k] = \{1, 2, \ldots, k\}$. Knopfmacher et al.\ \cite{K} introduced a class of words which they called staircase words. These are words $w_1\cdots w_n\in[k]^n$ which satisfy $|w_{i+1}-w_i|\leq 1$, for every $1\leq i\leq n-1$. Inspired by this, we defined staircase graph words in \cite{FM}, which are graph labelings with values in $[k]$, such that the labels of adjacent vertices differ by at most $1$. Under this definition, staircase words of length $n$ are staircase graph words of the path graph $P_n$. In \cite{FM} we studied the number of staircase graph words of the grid graph, the rectangle-triangular graph and the king's graph, all of size $2\times n$. This work is concerned with the path graph, but with an extended edge set, namely, for a positive integer $L$ we define $P_{n,L}$ to be the path graph of order $n$, in which vertices are adjacent if their distance is not greater than $L$ (see Figure \ref{fig;30}). In particular, $P_{n,1}$ coincides with $P_n$. If one wishes to avoid the notion of staircase graph words and adhere to the standard terminology of restricted words, staircase graph words of $P_{n,L}$ are words $w_1\cdots w_n\in[k]^n$ which satisfy $\max\{w_{i},\ldots,w_{i+L}\}-\min\{w_{i},\ldots,w_{i+L}\}\leq 1$, for every $1\leq i\leq n-L$. 

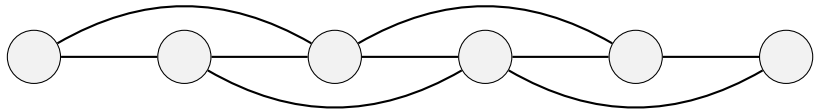
\begin{figure}[H]
\centering
\scalebox{1}{
\begin{tikzpicture}[shorten >=0pt,node distance=3cm,auto]
\node[draw,circle,inner sep=0.25cm, fill=black!5] at (0, 0)   (1) {};
\node[draw,circle,inner sep=0.25cm, fill=black!5] at (2, 0)   (2) {};
\node[draw,circle,inner sep=0.25cm, fill=black!5] at (4, 0)   (3) {};
\node[draw,circle,inner sep=0.25cm, fill=black!5] at (6, 0)   (4) {};
\node[draw,circle,inner sep=0.25cm, fill=black!5] at (8, 0)   (5) {};
\node[draw,circle,inner sep=0.25cm, fill=black!5] at (10, 0)   (6) {};
\path (1) edge[thick,-] node[above] {} (2);
\path (2) edge[thick,-] node[above] {} (3);
\path (3) edge[thick,-] node[above] {} (4);
\path (4) edge[thick,-] node[above] {} (5);
\path (5) edge[thick,-] node[above] {} (6);
\path (1) edge[thick,-, bend left] node[above] {} (3);
\path (3) edge[thick,-, bend left] node[above] {} (5);
\path (2) edge[thick,-, bend right] node[above] {} (4);
\path (4) edge[thick,-, bend right] node[above] {} (6);
\end{tikzpicture} }
\caption{The graph $P_{6,2}$}\label{fig;30}
\end{figure} \noindent

\section{Preliminaries}
The Chebyshev polynomials are two sequences of polynomials related to the cosine and sine functions, which find extensive use in approximation theory. They also emerge naturally in combinatorics, and, in particular, while studying staircase words (e.g., \cite{F, K}). The Chebyshev polynomials of the first kind, denoted by
$T_n(x)$, are defined by 
$T_{n}(\cos \theta )=\cos(n\theta)$, and satisfy the recursion 
\begin{align}
    T_0(x)&=1\nonumber\\
    T_1(x)&=x\nonumber\\
    T_{n+1}(x)&=2xT_n(x)-T_{n-1}(x).\nonumber
\end{align}
Similarly, the Chebyshev polynomials of the second kind, denoted by $U_n(x)$, are defined by 
$U_{n}(\cos\theta )\sin\theta=\sin((n+1)\theta)$, and satisfy the recursion 
\begin{align}
    U_0(x)&=1\nonumber\\
    U_1(x)&=2x\nonumber\\
    U_{n+1}(x)&=2xU_n(x)-U_{n-1}(x).\label{u1}
\end{align}
We shall make use of the following well-known Chebyshev polynomials identities (e.g., \cite[(2.4)]{K} and \cite{MH}):
\begin{align}
U_m(x)U_n(x)&=\frac{U_{m-n}(x)-xU_{m-n-1}(x)-U_{m+n+2}(x)+xU_{m+n+1}(x)}{2(1-x^{2})},\label{i3}\\
\left(x+\sqrt{x^{2}-1}\right)^{n}&=\frac{U_{n}(x)-U_{n-2}(x)}{2}+U_{n-1}(x)\sqrt{x^{2}-1},\label{i1} \\
U_{n}(x)^{2}&=1+U_{n-1}(x)U_{n+1}(x)\label{i5}.
\end{align}

\section{Main results}

We apply the kernel method, which is a widely used tool in combinatorics (e.g., \cite{MS1, P}). It involves the introduction of an additional variable. 

Let $k\geq 2$ and $L\geq 1$ be two integers and let $f(x)=f_{k,L}(x)$ be the generating function of the number of staircase graph words of $P_{n,L}$ over $[k]$. For $a_1,\ldots,a_L\in[k]$, we define  $f_{a_1,\ldots,a_L}(x)=f_{k,L,a_1,\ldots,a_L}(x)$ to be the generating function of the number of staircase graph words of $P_{n,L}$ over $[k]$, whose last $L$ letters are $a_1,\ldots,a_L$. We set $f_{a_1,\ldots,a_L}(x)=0$ if any of the numbers $a_1,\ldots,a_L$ does not belong to $[k]$. Furthermore, by definition of the staircase property, $f_{a_1,\ldots,a_L}(x)=0$ if 
$\max\{a_1,\ldots,a_L\}-\min\{a_1,\ldots,a_L\}>1$.

\begin{lemma}\label{ll1}
Let $I=\{0,1\}^L\setminus\{(0,\ldots,0),(1,\ldots,1)\}, F_{0,\ldots,0}(x,t)=\sum_{a=1}^{k}f_{a,\ldots,a}(x)t^{a-1}$, and, for each $(o_1,\ldots,o_L)\in I$, let $F_{o_1,\ldots,o_L}(x,t)=\sum_{a=1}^{k-1}f_{a+o_1,\ldots,a+o_L}(x)t^{a-1}$. Then
\begin{enumerate}
\item $f(x) = \sum_{i=0}^{L-1}(k+(k-1) (2^{i}-2))x^{i}+F_{0,\ldots,0}(x,1) +\sum_{o_1,\ldots,o_L}F_{o_1,\ldots,o_L}(x,1)$.
\item The $2^L-1$ functions $F_{0,\ldots,0}(x,t), F_{1,0,\ldots,0}(x,t),\ldots, F_{0,1,\ldots,1}(x,t)$ satisfy the following system of linear equations:
\begin{align}
F_{0,\ldots,0}(x,t)&=\frac{x^L(1-t^{k})}{1-t}+xF_{0,\ldots,0}(x,t) + xF_{1,0,\ldots,0}(x,t) + txF_{0,1,\ldots,1}(x,t),\nonumber\\
F_{1,\ldots,1,0}(x,t)&=\frac{x^L(1-t^{k})}{1-t} + \frac{x}{t}(F_{0,\ldots,0}(x,t)-f_{1,\ldots,1}(x)) + xF_{0,1,\ldots,1}(x,t),\nonumber\\
F_{0,\ldots,0,1}(x,t)&=\frac{x^L(1-t^{k})}{1-t} + x(F_{0,\ldots,0}(x,t)-t^{k-1}f_{k,\ldots,k}(x)),\nonumber
\end{align} 
and, for every $(o_1,\ldots,o_L)\in I\setminus\{(1,\ldots,1,0),(0,\ldots,0,1)\}$, 
\[F_{o_1,\ldots,o_L}(x,t)=\frac{x^L(1-t^{k})}{1-t} + xF_{0,o_1,\ldots,o_{L-1}}(x,t)+xF_{1,o_1,\ldots,o_{L-1}}(x,t).\]
\end{enumerate}
\end{lemma}

\begin{proof}
The first assertion is clear. Now,
for each $a\in[k]$, we have 
\begin{equation}\label{s500}
f_{a,\ldots,a}(x)=x^L+x\sum_{o\in\{-1,0,1\}}f_{a+o,a,\ldots,a}(x).
\end{equation}
Similarly, for each $a\in[k-1]$ and $(o_1,\ldots,o_L)\in I$, we have 
\begin{equation}\label{s501}
f_{a+o_1,\ldots,a+o_L}(x)=x^L+x\sum_{o\in \{0,1\}}f_{a+o,a+o_1,\ldots,a+o_{L-1}}(x).
\end{equation}
To obtain the system, we multiply \eqref{s500} and \eqref{s501} by $t^{a-1}$ and sum \eqref{s500} over $a\in[k]$ and \eqref{s501} over $a\in[k-1]$, respectively.
\end{proof}

In the next lemma we list the necessary properties of the system derived above. We omit its proof, which is left as an exercise to the interested reader. If $v$ is a vector, we write $v^T$ for the transpose of $v$.

\begin{lemma}\label{lemma10}
Let $A(x,t) = I - B(x,t)$, where $B(x,t)=(b_{ij}(x,t))$ is the square matrix of size $2^L-1$ given by
\[
b_{ij}(x,t) = 
\begin{cases}
tx,& \textnormal{if $(i,j)=(1,2^L-1)$}; \\
x, & \textnormal{if $(i,j)=(2^{L-1},2^L-1)$}\\
&\textnormal{or $j<2^L-1$ and $i = \left\lceil \frac{j}{2}\right\rceil$ or $i = \left\lceil \frac{j}{2}\right\rceil+2^{L-1}$};\\
\frac{x}{t}, & \textnormal{if $(i,j)=(2^{L-1},1)$};\\
0, &\textnormal{otherwise}.\\
\end{cases} \]
Let $q(x,t), r(x,t), u(x,t)$, and $v(x,t)$ be four functions and let 
\begin{align}
b(x,t) &= (q(x,t),r(x,t),\ldots,r(x,t))^T,\nonumber\\
b'(x,t) &= (0,\ldots,0,u(x,t),v(x,t),0,\ldots,0)^T,\nonumber
\end{align} be two vectors of length $2^L-1$, where $u(x,t)$ and $v(x,t)$ are at the $2^{L-1}$th and $2^{L-1}+1$th coordinates of $b'(x,t)$, respectively.
\begin{enumerate}
\item The determinant of $A(x,t)$ is given by $-K(x,t)/t$, where \[K(x,t)=x^{L}t^{2}+\left(\frac{x-x^{2L}}{1-x}-1\right)t+x^{L}.\] In particular, one of the two roots of $K(x,t)$ is $t_1 =\phi + \sqrt{\phi^2 - 1}$, where $\phi = \left(1-\frac{x-x^{2L}}{1-x}\right)/2x^L$.

%\[t_1=\frac{1-\sum_{i=1}^{2L-1}x^{i}+\sqrt{\left(\sum_{i=1}^{2L-1}x^{i}-1\right)^2-4x^{2L}}}{2x^L}.\]
\item We have
\begin{align}
c_{1}(x,t)&=\frac{t\frac{1-x^L}{1-x}\left(\sum_{i=1}^{L-1}(q(x,t)-(t+1)r(x,t))x^{i}\right)}{K(x,t)},\nonumber\\	
c'_{1}(x,t)&=-\frac{tx^{L-1}\left(u(x,t)+tv(x,t)+(tu(x,t)+v(x,t))\frac{x-x^L}{1-x}\right)}{K(x,t)}.\nonumber	
\end{align}
\item  Let $c(x,t)=(c_1(x,t),\ldots,c_{2^L-1}(x,t))^T$ and $c'(x,t)=(c'_1(x,t),\ldots,c'_{2^L-1}(x,t))^T$ be the unique vectors satisfying $A(x,t)c(x,t)=b(x,t)$ and $A(x,t)c'(x,t)=b'(x,t)$, respectively. Then
\begin{align}
\sum_{i=1}^{2^L-1}c_i(x,t) &=-\frac{\alpha(x,t)t^2+\beta(x,t)t+\gamma(x,t)}{K(x,t)},\nonumber\\
\sum_{i=1}^{2^L-1}c'_i(x,t) &=-\frac{\delta(x,t)t^{2}+\epsilon(x,t)t-\zeta(x,t)}{K(x,t)}\nonumber,
\end{align}
where 
\begin{align}
\alpha(x,t)&=r(x,t)x\sum_{i=0}^{L-2}(2^{i+1}-1)x^{i},\nonumber\\
\beta(x,t)&=r(x,t)x^{L}\sum_{i=0}^{L-2}\left(\sum_{j=i}^{L-2}(2^{j+1}-1)\right)x^{i}+\sum_{i=0}^{L-1}\left(q(x,t)+r(x,t)\sum_{j=0}^{L-2}(2^{j+1}-1)\right)x^{L-1-i},\nonumber\\
\gamma(x,t)&=\sum_{i=1}^{L-1}\left(i(q(x,t)-r(x,t))x^{2L-1-i}+(iq(x,t)+(2^{i}-i-1)r(x,t))x^{i}\right),\nonumber\\
\delta(x,t)&=v(x,t)x^{L-1},\nonumber\\
\epsilon(x,t)&=v(x,t)\frac{1-x^{2L-1}}{1-x}+v(x,t)x^{L-1}+u(x,t)\frac{1-x^L}{1-x},\nonumber\\
\zeta(x,t)&=v(x,t)x^{L}\frac{1-x^{L-1}}{1-x}.\nonumber
\end{align}
\end{enumerate}
\end{lemma}

From the previous lemma we derive the desired generating function $f(x)$. 

\begin{theorem}
Let $\phi = \left(1-\frac{x^{2L}-x}{x-1}\right)/2x^{L}$. We have 
\begin{equation}\label{eq99}
f(x)=\frac{\beta}{1-2x-x^{L}}\left(1+(k-3)x+x\frac{r_{1}(x)U_{k-1}(\phi)+r_{2}(x)U_{k}(\phi)+r_{3}(x)}{r_{4}(x)U_{k-1}(\phi)+r_{5}(x)U_{k}(\phi)+r_{3}(x)}\right),\nonumber
\end{equation} where $\beta=(1-x^{L})/(1-x)$ and
\begin{align}
r_{1}(x)&=(4\phi^{2}+2\phi-1)\beta^{2}-4(2\phi^{2}-1)\beta+4\phi^{2}-2(\phi+1)\nonumber\\
r_{2}(x)&=-(1+2\phi)\beta^{2}+4\beta\phi-2(\phi-1)\nonumber\\
r_{3}(x)&=-2(\phi-1)(\beta-1)-\beta^{2}\nonumber\\
r_{4}(x)&=(2\phi^{2}-1)\beta^{2}-2(2\phi^{2}-\phi-1)\beta+2\phi(\phi-1)\nonumber\\
r_{5}(x)&=2(\phi-1)(\beta-1)-\beta^{2}\phi.\nonumber
\end{align}
\end{theorem}

\begin{proof}
In the notation of Lemma \ref{lemma10}, let 
\begin{align}
q(x,t)&=\frac{x^{L}(1-t^{k})}{1-t},&&r(x,t)=\frac{x^{L}(1-t^{k-1})}{1-t},\nonumber\\
u(x,t)&=-\frac{xf_{1,\ldots,1}(x)}{t},&&v(x,t)=-xt^{k-1}f_{1,\ldots,1}(x).\nonumber
\end{align}
First, notice that, due to symmetry, $f_{1,\ldots,1}(x)=f_{k,\ldots,k}(x)$. By the first part of Lemma \ref{ll1}, 
\begin{align}
f(x)& = \sum_{i=0}^{L-1}(k+(k-1) (2^{i}-2))x^{i}+F_{0,\ldots,0}(x,1) +\sum_{o_1,\ldots,o_L}F_{o_1,\ldots,o_L}(x,1)\nonumber\\
&=\sum_{i=0}^{L-1}(k+(k-1) (2^{i}-2))x^{i}+\lim_{t\to 1}\sum_{i=1}^{2^L-1}(c_i(x,t)+c'_i(x,t)).\nonumber
\end{align} Application of the third part of Lemma \ref{lemma10} on the last expression gives us the formula for $f(x)$. Finally, by setting $t=t_1$ we eliminate $K(x,t)$ in the second part of Lemma \ref{lemma10} and solving for $f_{1,\ldots,1}(x)$ gives us 
\begin {align}
f(x)&=\sum_{i=0}^{L-1}(k+(k-1) (2^{i}-2))x^{i}
\nonumber\\&+ \frac{2xf_{1,\ldots,1}(x)+(2^{L}-2-(2^{L}-1)k)x^{L}+\sum_{i=1}^{L-1}(2^{i}-2-(2^{i}-1)k)x^{i+L}}{x^{L}+2x-1},\nonumber
\end{align} where
\[f_{1,\ldots,1}(x)=-\frac{((t_{1}-1+\beta)t_{1}^{k}-t_{1}(t_{1}\beta-t_{1}+1))\beta}{((t_{1}-1+\beta)t_{1}^{k}+t_{1}\beta-t_{1}+1)(1-t_{1})}. \]
Now, using \eqref{i1} and rewriting the denominator as $A+B\sqrt{\phi^2-1}$, we multiply the numerator and the denominator by $A-B\sqrt{\phi^2-1}$, to eliminate the square root in the denominator. Using \eqref{u1} and \eqref{i3} repeatedly, we obtain \[\frac{r_{1}(x)U_{k-1}(\phi)+r_{2}(x)U_{k}(\phi)+r_{3}(x)}{r_{4}(x)U_{k-1}(\phi)+r_{5}(x)U_{k}(\phi)+r_{3}(x)}.\qedhere\]
\end{proof}

\begin{example}
Taking $L=1$, the generating function \eqref{eq99} reduces to the one obtained in \cite[Theorem 2.2]{K}, which is given by \begin{equation}\label{ww1}
1+\frac{kx}{1-3x}-\frac{2x^{2}}{(1-3x)^{2}}\frac{U_{k}\left(\frac{1-x}{2x}\right)-U_{k-1}\left(\frac{1-x}{2x}\right)-1}{U_{k}\left(\frac{1-x}{2x}\right)}.
\end{equation} Indeed, in this case
$\phi= (1-x)/2x, 
\beta=1$, and 
\[f_{1,\ldots,1}(x)=\frac{U_{k-1}(\phi)+U_{k}(\phi)-1}{U_{k-1}(\phi)-\phi U_{k}(\phi)-1}.\] Thus,
\begin{align}
f(x)&=\frac{1+(k-3)x+xf_{1,\ldots,1}(x)}{1-3x}\nonumber\\&=1+\frac{kx}{1-3x}+\frac{x}{1-3x}\frac{U_{k-1}(\phi)+U_{k}(\phi)-1}{U_{k-1}(\phi)-\phi U_{k}(\phi)-1}.\label{w1}
\end{align}
Hence, to see that \eqref{ww1} and \eqref{w1} are equal, it suffices to show that \[U_{k-1}(x)^2-2xU_{k-1}(x)U_{k}(x)+U_{k}(x)^2=1.\] This identity is easily seen to hold true by using \eqref{u1} together with \eqref{i5}.
\end{example}

\begin{example}\label{e10}
Let $L=2$. In this case, $\phi = -(x^{3}+x^{2}+x-1)/2x^{2}, \beta=x+1$, and \[f_{1,\ldots,1}(x)=\frac{(2x^{4}+2x^{3}-2)U_{k-1}(\phi)+2x^{3}U_{k}(\phi)+2x}{(x^{4}+2x^{3}-1)U_{k-1}(\phi)+(x^{3}+x^{2}-x+1)U_{k}(\phi)+2x}.\]
In Table \ref{tab1} we list the generating functions for several values of $k$. 
\begin{table}[H]
\centering
{\renewcommand{\arraystretch}{1.6}
\begin{tabular}{|c|c|}\hline
$k$ & $f_{k,2}(x)$ \\
\hline\hline
 $2$ & $-\frac{1}{2x-1}$ \\ \hline
$3$ & $-\frac{x^{2}+x+1}{x^{3}+2x-1}$ \\ \hline
$4$ & $-\frac{x^{4}+3x^{3}+x^{2}-x-1}{x^{4}+x^{3}+x^{2}-3x+1}$ \\\hline
$5$ & $-\frac{2 x^{6}+5 x^{5}+5 x^{4}+5 x^{3}+x^{2}-2 x-1}{x^{6}+2 x^{5}+2 x^{4}+x^{2}-3 x+1}$ \\\hline
$6$ & $-\frac{3 x^{7}+6 x^{6}+5 x^{5}+x^{4}-8 x^{3}-5 x^{2}+2 x+1}{x^{7}+2 x^{6}+x^{5}-x^{4}-2 x^{3}-3 x^{2}+4 x-1}$ \\
\hline
\end{tabular}  
\caption{The generating functions of the number of staircase graph words of $P_{n, 2}$ over $[k]$, for $k=2,3\ldots,6$.}\label{tab1}}
\end{table}

\end{example}

%\paragraph{{\bf Acknowledgments}} The author wishes to thank Toufik Mansour for many helpful conversations. 

\end{document}